\documentclass{amsart}
\usepackage{latexsym,amsxtra,amscd,ifthen}
\usepackage{mathrsfs}
\usepackage{amsfonts}
\usepackage{txfonts}
\usepackage{amssymb}
\usepackage[arrow,matrix]{xy}
\usepackage{amsmath,amssymb,amscd,bbm,amsthm,mathrsfs,dsfont,cases}
\usepackage{amsmath,amscd}
\usepackage{amsfonts}
\usepackage{verbatim}
\usepackage{amsmath}
\usepackage{amsthm}
\usepackage{amssymb}
\usepackage{url}

\theoremstyle{plain}

\newtheorem{theorem}{Theorem}
\newtheorem{lemma}[theorem]{Lemma}
\newtheorem{proposition}[theorem]{Proposition}

\numberwithin{theorem}{section}
\numberwithin{equation}{theorem}
\theoremstyle{definition}

\newtheorem{definition}[theorem]{Definition}
\newtheorem{example}[theorem]{Example}
\newtheorem{remark}[theorem]{Remark}
\newtheorem{question}[theorem]{Question}
\newtheorem*{question*}{Question}

\DeclareMathOperator{\Ext}{Ext}

\DeclareMathOperator{\Mod}{Mod}

\DeclareMathOperator{\Aut}{Aut}

\begin{document}
\title
{Nakayama automorphism of a class of graded algebras}
\author{J.-F. L\"u}
\address{(L\"u) Department of Mathematics,
Zhejiang Normal University, Jinhua 321004, China}
\email{jiafenglv@zjnu.edu.cn}

\author{X.-F. Mao}
\address{(Mao) Department of Mathematics,
Shanghai University, Shanghai 200444, China}
\email{xuefengmao@shu.edu.cn}

\author{J.J. Zhang}
\address{(Zhang) Department of Mathematics,
 University of Washington, Seattle, WA 98195, USA}
\email{zhang@washington.edu}

\begin{abstract}
The Nakayama automorphism of a class of connected graded 
Artin-Schelter regular algebras is calculated explicitly. 
\end{abstract}

\subjclass[2010]{16E65, 16W50,  16E10}

\keywords{Nakayama automorphism, 
Artin-Schelter regular algebra}

%\thanks{ }

\maketitle
%\tableofcontents

\setcounter{section}{-1}
\section{Introduction}
\label{xxsec0} 
Throughout let $\Bbbk$ denote a base field and assume that 
$\mathrm{char}\,\Bbbk\neq 2$. All vector spaces and algebras are 
over $\Bbbk$.

We start with the definition of a class of algebras, denoted by 
$A(Q,C,s)$, as follows: 

Let  $Q :=(q_{ij})_{n\times n}$ and $C:=(c_{ij})_{n\times n}$ be two
$n\times n$-matrices in $M_n(\Bbbk)$ such that
\begin{equation}
\label{E0.0.1}\tag{E0.0.1}
q_{ii}=1, \quad q_{ji}=q^{-1}_{ij}, \quad {\text{for all $i$ and $j$,}}
\end{equation}
and
\begin{equation}
\label{E0.0.2}\tag{E0.0.2}
c_{ii}=0, \quad
c_{ji}=-q_{ij}^{-1}c_{ij}=-q_{ji}c_{ij}, \quad {\text{for all $i$ and $j$.}}
\end{equation}
We may call $C$ a {\it skew $Q$-symmetric} matrix.
Let $s$ be an integer such that $1\leq s\leq n$ and
let $\Omega$ denote $\sum_{i=1}^s t_i^2$, considered as an element in the free
algebra $\Bbbk\langle t_1,t_2,\cdots,t_n\rangle$. Define 
$$A(Q,C,s)=\Bbbk\langle t_1,\cdots,t_n\rangle/
(t_j t_i-q_{ij}t_it_j-c_{ij}\Omega,\forall \; i,j).$$
If $c_{ij}=0$ for all $i,j$, it becomes the usually
skew polynomial ring $\Bbbk_{q_{ij}}[t_1,\cdots,t_n]$.
So we will only consider the case when some of $c_{ij}$s are nonzero. 
By carefully choosing the parameters $Q$ and $C$, we can obtain plenty of 
interesting algebras, though all of them are Ore extensions
(at least when $s<n$, see Proposition \ref{yypro1.6}). 
Some examples are given in Section 2. 
The algebra $A(Q,C,s)$ can also be viewed as a deformation of
$\Bbbk_{q_{ij}}[t_1,\cdots,t_n]$ for some special values 
$(q_{ij})_{n\times n}$. 

By definition, $A(Q,C,s)$ is a quadratic algebra. The only result of 
this paper is to give some conditions on $Q$ and $C$ such that $A(Q,C,s)$ is 
Koszul, strongly noetherian, Cohen-Macaulay, Artin-Schelter regular, and 
Auslander regular of global dimension $n$ (and it is also a graded skew 
Clifford algebra); and when in this situation, we give an explicit 
formula of the Nakayama automorphism of the algebra 
$A(Q,C,s)$ in terms of a linear map associated to the basis 
$\{t_1,t_2\cdots,t_n\}$. We omit the definitions of some 
standard terms mentioned as above; these can be found in, 
for example, \cite{ASZ, CV1, Le, RRZ1, RRZ2, SZ, ZZ}. However, we will
recall the definition of the Nakayama automorphism in Definition \ref{yydef0.3} 
below because it is the objective of this paper.

For simplicity, $A(Q,C,s)$ is sometimes denoted by $A(Q,C)$ or $A$.
Consider the following conditions for $Q$ and $C$:
\begin{align}
\label{E0.0.3}\tag{E0.0.3}
q_{j\alpha}^2&=q_{j1}^2,
\quad {\text{for all $j$ and all $\alpha=2,\cdots,s$}}\\
\label{E0.0.4}\tag{E0.0.4}
q_{i\alpha}q_{j\alpha}c_{ij}&=q_{1\alpha}^2c_{ij}, 
 {\text{for all $\alpha$ and all $i,j\neq \alpha$}}\\
\label{E0.0.5}\tag{E0.0.5}
q_{i\alpha}c_{i\alpha}&=c_{i\alpha},
\quad\; {\text{for all $\alpha\leq s$ and all $i$}}\\
\label{E0.0.6}\tag{E0.0.6}
c_{\alpha j}c_{kl}-q_{jl}c_{\alpha l}c_{kj}
+q_{lk}q_{jk}c_{\alpha k}c_{lj}&=0,
\qquad {\text{for all $\alpha\leq s$ and all $j<l<k$}}\\
\label{E0.0.7}\tag{E0.0.7}
c_{\alpha k}c_{lj}-q_{lk}c_{\alpha l}c_{kj}
+q_{jl}q_{jk}c_{\alpha j}c_{kl}&=0,
\qquad {\text{for all $\alpha\leq s$ and all $j<l<k$}}\\
\label{E0.0.8}\tag{E0.0.8}
\det(I_{s\times s}-(c_{ij})_{s\times s})&\neq 0.
\end{align}

\begin{remark}
\label{yyrem0.1} It follows from \eqref{E0.0.1}, \eqref{E0.0.2} and 
\eqref{E0.0.5} that $C_s:=(c_{ij})_{s\times s}$ is skew symmetric, though 
the bigger matrix $C:=(c_{ij})_{n\times n}$ may not be, see
Example \ref{yyex2.1}. Hence \eqref{E0.0.8} is equivalent to
\begin{equation}
\notag
%\label{E0.0.9}\tag{E0.0.9}
\det(I_{s\times s}+(c_{ij})_{s\times s})\neq 0.
\end{equation}
\end{remark}

\begin{theorem}
\label{yythm0.2} Assume ${\rm char}\; \Bbbk\neq 2$.
Suppose $Q$ and $C$ satisfy \eqref{E0.0.1}-\eqref{E0.0.8}. Then
\begin{enumerate}
\item[(1)]
$A(Q,C)$ is a Koszul Artin-Schelter regular algebra of global 
dimension $n$ with Hilbert series $\frac{1}{(1-t)^n}$.
\item[(2)]
It is a strongly noetherian, Cohen-Macaulay, and Auslander regular
domain.
\item[(3)]
The Nakayama automorphism of $A(Q,C)$ is determined by
$$\mu_{A(Q,C)}: t_j\mapsto \begin{cases}
\sum\limits_{i=1}^s\sum\limits_{l=1}^sa_{il}c_{lj}(1+q_{j1}^2)
(\prod\limits_{w=1}^nq_{wj})t_i 
+(\prod\limits_{w=1}^nq_{wj})t_j, \,\,\text{if}\,\, j> s\\
\sum\limits_{i=1}^sb_{ij}q_{j1}^2
(\prod\limits_{w=1}^nq_{wj})t_i, \qquad \qquad\qquad\qquad\qquad
\text{if}\,\,  j\le s.
\end{cases} $$
where $(a_{ij})_{s\times s}= [I_{s\times s}-(c_{ij})_{s\times s}]^{-1}$ and 
$(b_{ij})_{s\times s}= [I_{s\times s}-(c_{ij})_{s\times s}]^{-1}
[I_{s\times s}+(c_{ij})_{s\times s}]$.
\end{enumerate}
\end{theorem}

As a consequence, $A(Q,C)$ is a graded skew Clifford algebra in the
sense of Cassidy-Vancliff \cite{CV1, CV2}, see Proposition \ref{yypro1.5}. 
Assisted by any computing software, say Maple, after finding a solution 
to system \eqref{E0.0.1}-\eqref{E0.0.8} (which consists of quite 
naive equations), one can construct immediately an Artin-Schelter 
regular algebra of dimension $n$. This is similar to the ideas of 
Cassidy-Vancliff in \cite{CV1, CV2}. 

As promised we recall the definition of the Nakayama automorphism,
see \cite{RRZ1} for more information.

\begin{definition}
\label{yydef0.3}
Let $A$ be an algebra over $\Bbbk$. Let $A^e=A\otimes_{\Bbbk} A^{op}$.
\begin{enumerate}
\item[(1)]
$A$ is called {\it skew Calabi-Yau} if
\begin{enumerate}
\item[(i)]
$A$ is {\it homologically smooth}, that is, $A$ has a projective resolution
in the category $A^e$-$\Mod$ that has finite length and each term in
the projective resolution is finitely generated, and
\item[(ii)]
there is an integer $d$ and an algebra automorphism $\mu$ of $A$ such
that
\begin{equation}
\label{E0.3.1}\tag{E0.3.1}
\Ext^i_{A^e}(A, A^e)\cong \begin{cases} 0 & i\neq d\\
{^1A^\mu} & i=d, \end{cases}
\end{equation}
as $A$-bimodules, where $1$ denotes the identity map of $A$.
\end{enumerate}
\item[(2)]
If \eqref{E0.3.1} holds for some algebra automorphism $\mu$ of $A$,
then $\mu$ is called the {\it Nakayama automorphism} of $A$, and
is usually denoted by $\mu_A$. Note that $\mu_A$ (if exists) is
unique up to inner automorphisms of $A$.
\item[(3)]
$A$ is called {\it Calabi-Yau}  if it is skew
Calabi-Yau and $\mu_A$ is inner, or equivalently, $\mu_A$ can be
chosen to be the identity map after changing a bimodule generator of 
${^1A^\mu}$.
\end{enumerate}
\end{definition}

When $A$ is a connected graded domain, every unit of $A$ is a nonzero
scalar, which is central. In this case, every inner automorphism 
of $A$ is the identity map and the Nakayama automorphism of 
$A$ (if exists) is unique.

By \cite[Lemma 1.2]{RRZ1}, a connected graded algebra is 
skew Calabi-Yau if and only if it is Artin-Schelter regular.
The Nakayama automorphism of an Artin-Schelter regular algebra
(or more generally a 
skew Calabi-Yau algebra) is an important invariant, which is closely related to 
the study of 
\begin{enumerate}
\item[(1)]
rigid dualizing complexes \cite{V1} and the Calabi-Yau 
property \cite{RRZ1, RRZ2};
\item[(2)]
the Poincar{\'e} duality (or Van den Bergh duality) 
between Hochschild homology and 
Hochschild cohomology \cite{V2};
\item[(3)]
the automorphism group of Artin-Schelter regular algebras 
and group/Hopf algebra actions on Artin-Schelter regular 
algebras \cite{LMZ};
\item[(4)]
locally nilpotent derivations and cancellation problem 
of Artin-Schelter regular algebras \cite{LMZ}.
\end{enumerate}
Several authors have been studying the Nakayama automorphism of some 
special classes of algebras during the last few years. Yekutieli 
gave an explicit formula of the Nakayama automorphism of the universal 
enveloping algebra $U(L)$ of a finite dimensional Lie algebra $L$
\cite{Ye}. Rogalski-Reyes-Zhang proved several homological identities about 
the Nakayama automorphism in \cite{RRZ1, RRZ2}.  Liu-Wang-Wu studied 
the Nakayama automorphism for Ore extensions in \cite{LWW}; 
in particular, they gave a description of the Nakayama automorphism 
of the Ore extension algebra $A[x;\sigma,\delta]$. Zhu-Van Oystaeyen-Zhang 
computed the Nakayama automorphism of a trimmed double Ore extension 
of a Koszul Artin-Schelter regular algebra in \cite{ZVZ}. The authors 
of the present paper studied group actions and Hopf algebra actions 
on Artin-Schelter regular algebras of global dimension three in 
connection with the Nakayama automorphism \cite{LMZ}. In general, the 
Nakayama automorphism is a subtle invariant and difficult to compute.
Therefore it is definitely worth understanding and calculating the 
Nakayama automorphism for more examples.

Note that the methods introduced in the papers \cite{RRZ1, RRZ2, LWW} 
does not apply in our situation. We use another well-known 
method coming from Cassidy-Vancliff's paper \cite{CV1}, by constructing 
a sequence of regular normalizing elements. Consequently, the Koszul 
dual of $A(Q,C)$ is a complete intersection in the sense of 
\cite{CV1, CV2, KKZ}. The algebras appeared in Theorem \ref{yythm0.2}
are very easy in terms of generators and relations. It would be 
interesting to work out further properties concerning group or Hopf 
algebra actions and universal quantum linear group coactions (as in
the work \cite{WW}) on $A(Q,C)$, and invariants, such as, the center, the 
automorphism group, and the Makar-Limanov invariant, of $A(Q,C)$.

We finish the introduction with the following question.

\begin{question}
\label{yyque0.4}
What is the explicit formula for the Nakayama automorphism 
of a graded Artin-Schelter regular skew Clifford algebra 
in the sense of \cite{CV1, CV2}?
\end{question}

\section*{Acknowledgments}
J.-F. L\"u  was supported by NSFC (Grant Nos.11571316 and 11001245) and Natural
Science Foundation of Zhejiang Province (Grant No. LY16A010003).
X.-F. Mao was
supported by NSFC (Grant No.11001056), the Key Disciplines of Shanghai Municipality
(Grant No.S30104) and the Innovation Program of Shanghai Municipal
Education Commission (Grant No.12YZ031).
J.J. Zhang was supported by the US National Science
Foundation (NSF grant No. DMS 1402863).

\section{The proof of Theorem \ref{yythm0.2}}
\label{yysec1}
In this section, we study the algebra $A(Q,C)$ defined as in 
the introduction. Throughout this section we assume 
that 
$${\text{{\bf the parameters $(Q,C)$ satisfy conditions 
\eqref{E0.0.1}-\eqref{E0.0.8}. }}}$$

By definition, the algebra $A(Q,C)$ is a quadratic algebra, and its 
Koszul dual is 
$$A^!(Q,C):=\Bbbk \langle x_1,\cdots,x_n\rangle/(Rel)$$
where $\{x_1,\cdots,x_n\}$ is the dual basis of $\{t_1,\cdots,t_n\}$
and the ideal $(Rel)$ is generated by the following relations.
\begin{align}
\label{E1.0.1}\tag{E1.0.1}
x_j x_i+q_{ij}^{-1} x_i x_j&=0, \qquad \quad {\text{for all $1\le i< j\le n$}},\\
\label{E1.0.2}\tag{E1.0.2}
x_i^2&=0, \qquad \quad {\text{for all $i>s$}},\\
\label{E1.0.3}\tag{E1.0.3}
x_i^2-x_1^2&=0, \qquad \quad {\text{for all $1<i\leq s$}},\\
\label{E1.0.4}\tag{E1.0.4}
x_1^2&= \sum_{i<j} q_{ij}^{-1}c_{ij} x_i x_j
(=-\sum_{i<j} c_{ji} x_i x_j).
\end{align}
Let $p_{ij}=-q_{ij}^{-1}$ and define 
\begin{align}
\label{E1.0.5}\tag{E1.0.5}
T_1=\Bbbk_{p_{ij}}[x_1,\cdots,x_n]
(:=\Bbbk\langle x_1,\cdots,x_n\rangle/(x_jx_i-p_{ij}x_ix_j,\forall \; i<j)).
\end{align}
We can construct a sequence of regular normal elements. Since
$T_1$ is a skew polynomial ring, it is Koszul and Artin-Schelter 
regular of global dimension $n$. It is well-known that the monomials 
$\{x_1^{d_1}\cdots x_n^{d_n} \mid d_s\geq 0\}$ form a $\Bbbk$-linear 
basis and the Hilbert series of $T_1$ is $\frac{1}{(1-t)^{n}}$. 
Since $x_i^2$ are normal and since condition \eqref{E0.0.3}, one can 
check, by using the monomial basis, that $\{x_1^2-x_i^2\}_{i=2}^{s}
\cup \{ x_p^2\}_{p=s+1}^n$ form a regular normal sequence of $T_1$.
By factoring out this sequence of regular normal elements in $T_1$, we let
\begin{equation}
\label{E1.0.6}\tag{E1.0.6}
T_2=T_1/(\{x_1^2-x_i^2\}_{i=2}^{s}\cup \{ x_p^2\}_{p=s+1}^n).
\end{equation}
Then $T_2$ is an Artin-Schelter Gorenstein algebra of injective dimension 
one and its Hilbert series is
$$H_{T_2}(t)=\frac{(1-t^2)^{n-1}}{(1-t)^n}=\frac{(1+t)^{n-1}}{(1-t)}.$$
Note that $A^!(Q,C)$ is a factor ring of $T_2$.
Next we define some elements in $T_1$ (or $T_2$ or $A^!(Q,C)$). 
For any $1\leq \alpha\leq n$, let
$$y_{\alpha}:=\sum_{i=1}^n c_{\alpha i} x_i=\sum_{i\neq \alpha}
c_{\alpha i} x_i,$$
$$W_\alpha:=y_{\alpha} x_{\alpha}
=\sum_{i<\alpha} c_{\alpha i} x_i x_{\alpha}+
\sum_{j>\alpha} c_{j \alpha } x_{\alpha}x_j,$$
$$M_{\alpha}:=\sum_{i<j, i\neq \alpha,j\neq \alpha} c_{ji} x_ix_j,$$
$${\mathbb W}:=\sum_{\alpha>s}W_{\alpha},$$
and
$$\phi:=x_1^2+\sum_{i<j} c_{ji} x_i x_j.$$
Note that the algebras $T_1$, $T_2$ and $A^!(Q,C)$ and the above 
elements are well-defined without $(Q,C)$ satisfying 
\eqref{E0.0.1}-\eqref{E0.0.8}, but we always assume 
\eqref{E0.0.1}-\eqref{E0.0.8} in this section.

The following is an easy observation.

\begin{lemma}
\label{yylem1.1}
Parts {\rm{(1,2)}} hold
in the algebra $T_1$ and part {\rm{(3)}} holds in $T_2$.
\begin{enumerate}
\item[(1)]
$x_{\alpha} y_{\alpha} =-y_{\alpha} x_{\alpha}$
and $x_{\alpha} W_{\alpha}=-W_{\alpha} x_{\alpha}$
for all $\alpha\leq s$.
\item[(2)]
$x_{\alpha} M_{\alpha}=q_{1\alpha}^{-2} M_{\alpha} x_{\alpha}$
for all $\alpha$, and $x_{\alpha} M_{\alpha}=M_{\alpha} x_{\alpha}$
for all $\alpha\leq s$.
\item[(3)]
$W_{i}x_i=x_i W_{i}=0$ for all $i>s$.
\end{enumerate}
\end{lemma}

\begin{proof} Each follows from a direct computation under the 
conditions \eqref{E0.0.1}-\eqref{E0.0.5}.
\end{proof}

\begin{lemma}
\label{yylem1.2} 
The following hold in $T_2$.
\begin{enumerate}
\item[(1)]
$\sum_{j} c_{\alpha j} x_j M_j=\sum_{j} c_{\alpha j} M_j x_j=0$
for all $\alpha\leq s$.
\item[(2)]
If $i>s$, then $x_i \phi=q_{1i}^{-2} \phi x_i$.
\item[(3)]
If $i\le s$, then $(x_i+y_i) \phi=\phi (x_i-y_i)$.
\item[(4)]
$\phi$ is normal in $T_2$.
\end{enumerate}
\end{lemma}

\begin{proof} (1) 
We use \eqref{E0.0.7} in the last step of the following 
computation
$$\begin{aligned}
\sum_{j} c_{\alpha j}  x_j M_j &=
\sum_{j} c_{\alpha j}  x_j (\sum_{l<k, l\neq j, k\neq j} c_{kl} x_{l} x_k) \\
&=\sum_{j,l<k, l\neq j, k\neq j}c_{\alpha j}c_{kl} x_jx_k x_l \\
&=\sum_{j<l<k} c_{\alpha j}c_{kl}  x_j x_k x_l+
\sum_{l<j<k} c_{\alpha j}c_{kl}  x_jx_k x_l+
\sum_{l<k<j} c_{\alpha j}c_{kl}  x_jx_k x_l\\
&=\sum_{j<l<k} c_{\alpha j}c_{kl}  x_j x_k x_l+
\sum_{j<l<k} c_{\alpha l}c_{kj}  x_lx_k x_j+
\sum_{j<l<k} c_{\alpha k}c_{lj}  x_kx_l x_j\\
&=\sum_{j<l<k} (q_{jk}q_{jl}c_{\alpha j}c_{kl} -q_{lk}c_{\alpha l}c_{kj}+
c_{\alpha k}c_{lj})x_k x_lx_j \\
&=0.
\end{aligned}
$$
Similarly, we use \eqref{E0.0.6} in the last step of the following 
computation
$$\begin{aligned}
\sum_{j} c_{\alpha j}  M_j x_j&=
\sum_{j} c_{\alpha j}  (\sum_{l<k, l\neq j, k\neq j} c_{kl} x_{l} x_k)x_j \\
&=\sum_{j,l<k, l\neq j, k\neq j}c_{\alpha j}c_{kl} x_k x_lx_j \\
&=\sum_{j<l<k} c_{\alpha j}c_{kl}  x_k x_lx_j+
\sum_{l<j<k} c_{\alpha j}c_{kl}  x_k x_lx_j+
\sum_{l<k<j} c_{\alpha j}c_{kl}  x_k x_lx_j\\
&=\sum_{j<l<k} c_{\alpha j}c_{kl} x_k x_l x_j+
\sum_{j<l<k} c_{\alpha l}c_{kj} x_k x_jx_l +
\sum_{j<l<k} c_{\alpha k}c_{lj} x_l x_jx_k \\
&=\sum_{j<l<k} (c_{\alpha j}c_{kl} -q_{jl}c_{\alpha l}c_{kj}+
q_{lk}q_{jk} c_{\alpha k}c_{lj})x_k x_lx_j =0.
\end{aligned}
$$

(2) Since $i>s$, $x_{i}^2=0$
and $x_i W_i=W_i x_i=0$.
By Lemma \ref{yylem1.1}(2),  $x_i M_i=q_{1i}^{-2} M_i x_i$.
Hence
$$ \begin{aligned}
x_i \phi &= x_i( x_1^2+W_i+M_i)\\&=x_i x_1^2+ x_i M_i\\
&= q_{1i}^{-2} x_1^2 x_i+q_{1i}^{-2} M_i x_i\\
&=q_{1i}^{-2} (x_1^2+W_i+M_i)x_i\\&=q_{1i}^{-2}\phi x_i.
\end{aligned}
$$

(3)  Let $\alpha\leq s$. By using the fact that
$x_{\alpha}^2=x_1^2$ and that 
$\sum_{i<j} c_{ji} x_i x_j=W_l+M_l$ for all $l$, 
and by Lemma \ref{yylem1.1} and part (1), we have 
$$\begin{aligned}
(x_{\alpha}+y_{\alpha}) \phi&= (x_{\alpha}+y_{\alpha})
(x_{\alpha}^2+\sum_{i<j} c_{ji} x_i x_j)\\
&=x_{\alpha}^3 +y_{\alpha} x_{\alpha}^2 +x_{\alpha}(W_{\alpha}+M_{\alpha})
+\sum_{j} c_{\alpha j} x_j (W_{j}+M_{j})\\
%&=x_{\alpha}^3+W_{\alpha}x_{\alpha}+x_{\alpha}W_{\alpha}
%+x_{\alpha} M_{\alpha}
%+\sum_{j}c_{\alpha j} x_j W_j +\sum_{j}c_{\alpha j}x_j M_j\\
&=x_{\alpha}^3+W_{\alpha}x_{\alpha}+x_{\alpha}W_{\alpha}
+x_{\alpha} M_{\alpha}+\sum_{j}c_{\alpha j} x_j W_j \qquad {\text{part (1)}}\\
&=x_{\alpha}^3- x_{\alpha} W_{\alpha}+x_{\alpha}W_{\alpha}
+x_{\alpha} M_{\alpha} +\sum_{j}c_{\alpha j} x_j W_j \qquad {\text{Lemma \ref{yylem1.1}(1)}}\\
&=x_{\alpha}^3
+x_{\alpha} M_{\alpha} +\sum_{j\leq s}c_{\alpha j} x_j W_j \qquad
\qquad\qquad \qquad \;\; {\text{Lemma \ref{yylem1.1}(3).}}
\end{aligned}
$$
Similarly,
$$\begin{aligned}
\phi(x_{\alpha}-y_{\alpha}) &=
(x_{\alpha}^2+\sum_{i<j} c_{ji} x_i x_j)(x_{\alpha}-y_{\alpha})\\
&=x_{\alpha}^3 -x_{\alpha}^2y_{\alpha} +(W_{\alpha}+M_{\alpha})x_{\alpha}
-\sum_{j} c_{\alpha j} (W_{j}+M_{j}) x_j \\
&=x_{\alpha}^3-W_{\alpha} x_{\alpha}
+W_{\alpha}x_{\alpha}+M_{\alpha}x_{\alpha}
-\sum_{j} c_{\alpha j} W_{j} x_j-\sum_{j} c_{\alpha j} M_{j} x_j\\
&=x_{\alpha}^3+M_{\alpha}x_{\alpha}-\sum_{j\leq s} c_{\alpha j} W_{j} x_j\\
&=x_{\alpha}^3+x_{\alpha} M_{\alpha}+\sum_{j\leq s}c_{\alpha j} x_j W_j.
\end{aligned}
$$
So we have
$$ (x_{\alpha}+y_{\alpha})\phi=\phi(x_{\alpha}-y_{\alpha}).$$

(4) By Remark \ref{yyrem0.1}, $\det(I_{s\times s}-C_s)\neq 0$
if and only if $\det(I_{s\times s}+C_s)\neq 0$.
By hypothesis, $\det(I_{s\times s}-C_s)\neq 0$, 
whence $\{x_i-y_i\}_{i=1}^s \bigcup \{x_j\}_{j=s+1}^n$ is a basis 
of $V:=\bigoplus_{i=1}^n \Bbbk x_i$. Similarly, 
$\det(I_{s\times s}+C_s)\neq 0$ implies that $\{x_i+y_i\}_{i=1}^s
\bigcup \{x_j\}_{j=s+1}^n$ is a basis of $V$. In this case, $V\phi=\phi V$
by parts (2,3). Therefore $A\phi=\phi A$. 
\end{proof}
It is clear that $\sum_{i<j} c_{ji} x_i x_j=W_{\alpha}+M_{\alpha}$
for every $\alpha$.

\begin{lemma}
\label{yylem1.3}
The following hold in $A^!(Q,C)$.
\begin{enumerate}
\item[(1)]
$W_\alpha x_\alpha=x_{\alpha} W_\alpha=0$ for all $\alpha$.
\item[(2)]
$\sum_{\alpha=1}^n W_{\alpha}=2\sum_{i<j} c_{ji} x_i x_j
=\sum_{\alpha=1}^s W_{\alpha}+{\mathbb W}$.
\item[(3)]
$W_{\alpha}^2=0$ for all $\alpha>s$ and
${\mathbb W}^{m}=0$ for all $m> n-s$.
\item[(4)]
$x_1^4=-\frac{1}{2} \; {\mathbb W} x_1^2=-\frac{1}{2}\; x_1^2 {\mathbb W}$ 
and $x_1^{4m}=0$ for all $m> n-s$.
\item[(5)]
If $f\in A^!(Q,C)$ is a homogeneous element of degree $>4(n-s)n$,
then $f=0$. As a consequence, $A^!(Q,C)$ is finite dimensional.
\item[(6)]
The algebra $A^!(Q,C)$ is
Koszul and Frobenius of Hilbert series $(1+t)^n$.
\end{enumerate}
\end{lemma}

\begin{proof} (1) If $\alpha>s$, then the assertion follows from
Lemma \ref{yylem1.1}(3).
Now assume that $\alpha\leq s$.
It follows from Lemma \ref{yylem1.1}(1) that
$x_{\alpha} W_{\alpha}=- W_{\alpha} x_{\alpha}$.
By using Lemma \ref{yylem1.1}(2) and the fact that $p_{1\alpha}^2=1$ 
for all $\alpha\leq s$, we have the following computation in $A^!(Q,C)$:
$$\begin{aligned}
0&=x_1^2 x_{\alpha} -x_{\alpha} x_1^2\\
&=(-\sum_{i<j} c_{ji} x_i x_j) x_{\alpha}+
x_{\alpha}(\sum_{i<j} c_{ji} x_i x_j) \\
&=-(W_{\alpha}+M_{\alpha}) x_{\alpha}+x_{\alpha}
(W_{\alpha}+M_{\alpha} )\\
&=-W_{\alpha} x_{\alpha}-M_{\alpha}x_{\alpha}+x_{\alpha}W_{\alpha}
+x_{\alpha} M_{\alpha}\\
&=-W_{\alpha} x_{\alpha}-
M_{\alpha}x_{\alpha}-W_{\alpha}x_{\alpha}
+M_{\alpha} x_{\alpha}\\
&=-2W_{\alpha} x_{\alpha}.
\end{aligned}
$$
The assertion follows  since ${\rm char}\; \Bbbk\neq 2$.

(2) Clear.

(3) Note that $W_{j}=x_{j} f$ for some $f\in A^!(Q,C)$.
The assertions follow from the fact $x_{j}^2=0$ for all
$j>s$.

(4) We compute
$$\begin{aligned}
-2x_1^4 &= -2x_1^2 x_1^2=2(\sum_{i<j} c_{ji} x_i x_j) x_1^2
=(\sum_{\alpha=1}^s W_{\alpha}+{\mathbb W}) x_1^2\\
&=\sum_{\alpha=1}^s W_{\alpha}x_1^2 +{\mathbb W} x_1^2
=\sum_{\alpha=1}^s W_{\alpha}x_{\alpha}^2+ {\mathbb W}x_1^2
=0+ {\mathbb W} x_1^2= {\mathbb W} x_1^2.
\end{aligned}
$$
Similarly, $-2x_1^4= x_1^2 {\mathbb W}$.
So the first assertion follows. The second assertion
follows from the first one and part (3).

(5) It follows from part (4) and the relations of $A^!(Q,C)$ 
that $x_i^{4m}=0$ for all $m>n-s$.
So any  monomial of degree $>4(n-s)n$ is zero.
The assertion follows.

(6) By Lemma \ref{yylem1.2}(4), $\phi$ is normal in $T_2$.
By definition, $T_2/\phi T_2= T_2/(\phi)=A^!(Q,C)$, 
which is finite dimensional by part (5).
Let $K=\{ x\in T_2\mid \phi x=0\}$. Then $K$ is a
right graded $T_2$-module. Since $T_2$ is noetherian,
PI, Artin-Schelter Gorenstein of GK-dimension one, it is Cohen-Macaulay
with respect to GK-dimension \cite[Theorem 1.1]{SZ}.
In particular, any nonzero submodule of $T_2$ has
GK-dimension one by the definition of GK-Cohen-Macaulay.
Consider the following exact sequence
$$0\to K(2)\to T_2(2) \xrightarrow{l_{\phi}} T_2 \to
T_2/(\phi T_2)(=A^!(Q,C))\to 0,$$
where $K(2)$ and $T_2(2)$ are the degree shift by 2 of graded modules
$K$ and $T_2$ respectively,
and note that $A^!(Q,C)$ is finite dimensional. Since
$H_{T_2}(t)=\frac{(1+t)^n}{(1-t)}$, $\dim (T_2)_d=2^{n}$
for $d>n$. Thus the above exact sequence implies
that $K_{d+2}=0$ for all $d\gg 0$. By the GK-Cohen-Macaulay
property of $T_2$, $K=0$, and hence $\phi$
is left regular. Similarly, $\phi$ is right regular.
By \cite[Corollary 2.2]{CV1}, $T_2/(\phi)$ is Koszul and Frobenius
with Hilbert series $(1+t)^n$.
\end{proof}

Note that part (6) is a special case of \cite[Corollary 2.6]{CV1}.
We will need the following lemma in the proof of Theorem \ref{yythm0.2}.

\begin{lemma}
\label{yylem1.4}\cite[Lemma 1.5]{RRZ1}
Let $A$ be a noetherian connected graded Artin-Schelter Gorenstein algebra and
let $z$ be a homogeneous regular normal elements of positive degree
such that $\mu_A(z)=cz$ for some $c\in \Bbbk^{\times}:=\Bbbk\setminus \{0\}$. 
Let $\tau$  be in $\Aut(A)$ such that $za = \tau(a)z$ for all $a\in A$.
Then $\mu_{A/(z)}$ is equal to $\mu_A \circ \tau$ when
restricted to $A/(z)$.
\end{lemma}

Now, we can prove Theorem \ref{yythm0.2}.

\begin{proof}[Proof of Theorem \ref{yythm0.2}] 
(1) By Lemma \ref{yylem1.3}(6),
the algebra $A^!(Q,C)$ is Koszul and Frobenius with Hilbert series $(1+t)^n$.
By Koszul duality and \cite[Theorem 4.3 and Proposition 5.10]{Sm},
$A(Q,C)$ is Koszul, Artin-Schelter regular of global dimension 
$n$ with Hilbert series $\frac{1}{(1-t)^n}$. Thus part (1) follows.

(2)  By Lemmas \ref{yylem1.2} and \ref{yylem1.3},
$$\{x_i^2-x_1^2\}_{i=2}^{s}\cup \{x_p^2\}_{p=s+1}^n \cup \{\phi\}$$
is a sequence of regular normalizing elements of degree 2 in $T_1$.
By \cite[Corollary 1.4]{ST}, $A(Q,C)$ contains a sequence of regular 
normalizing elements of degree 2, say $\{f_1,\cdots,f_n\}$, 
such that $A(Q,C)/(f_1,\cdots,f_n)$ is finite dimensional.
Using \cite[Proposition 4.9(1)]{ASZ} repeatedly, $A(Q,C)$ is strongly 
noetherian. It is also easy to see that
$A(Q,C)$ has enough normal elements in the sense of \cite{Zh}.
Other assertions of part (2) follow from part (1) and \cite[Theorem 0.2]{Zh}.

(3) Let us work on the algebra $T_1$. By \cite[Proposition 4.1]{LWW}
or \cite[Example 5.5]{RRZ1},
the Nakayama automorphism of $T_1$ is determined by
$$\mu_{T_1}: x_j\mapsto (\prod_{w=1}^n p_{wj}) x_j=(-1)^n
(\prod_{w=1}^n q_{jw}) x_j$$
for all $j$.
Note that each $x_i^2-x_1^2$, for $i=2,\cdots,s$, is normal 
and 
$$\mu_{T_1}(x_i^2-x_1^2)=(\prod_{w=1}^{n} q_{1w})^2 (x_i^2-x_1^2)$$
by \eqref{E0.0.3}. The
conjugation by $x_i^2-x_1^2$ is the automorphism determined by
$$\tau_{x_i^2-x_1^2}: x_j \mapsto (x_i^2-x_1^2) x_j (x_i^2-x_1^2)^{-1}
=q_{ij}^2 x_j$$
for all $j$. Similarly, each $x_p^2$, for $p>s$, is normal and the
conjugation by $x_p^2$ is the automorphism determined by
$$\tau_{x_p^2}: x_j \mapsto (x_p^2) x_j (x_p^2)^{-1}
=q_{pj}^2 x_j$$
for all $j$. (Some details of checking the hypotheses of Lemma 
\ref{yylem1.4} is straightforward and omitted.)
Applying \cite[Lemma 1.5]{RRZ1} multiple times, we have
that the Nakayama automorphism of $T_2$ is
$$\mu_{T_2}=\mu_{T_1}\circ \prod_{i=2}^s \tau_{x_i^2-x_1^2}
\circ \prod_{p>s} \tau_{x_p^2},$$
or, after simplification, 
$$\mu_{T_2}: x_j\mapsto (-1)^n q_{j1}^2 (\prod_{w=1}^n q_{wj}) x_j$$
for all $j$.
By Lemma \ref{yylem1.2}, $\phi$ is normal in $T_2$ and 
$\mu_{T_2}(\phi)=(\prod_{w=1}^n q_{w1})^2 \phi$ by \eqref{E0.0.4}.
The conjugation by $\phi$ is determined by
$$\tau_{\phi}: x_j \mapsto \phi x_j \phi^{-1}.$$
By Lemma \ref{yylem1.2}(2,3),
$\tau_{\phi}$ sends $x_j$ to $q_{1j}^2 x_j$ when $j>s$ and
$x_j-y_j$ to $x_j+y_j$ when $j\leq s$. Therefore, by linear algebra, 
\begin{align*}
[I_{s\times s}-C_s]\left(
     \begin{array}{c}
     \tau_{\phi}(x_1)\\
     \tau_{\phi}(x_2)\\
     \vdots \\
     \tau_{\phi}(x_s) \\
     \end{array}
     \right)= [I_{s\times s}+C_s]\left(
     \begin{array}{c}
     x_1\\
     x_2\\
     \vdots \\
     x_s\\
     \end{array}
     \right)+ \left(
     \begin{array}{c}
     \sum\limits_{i=s+1}^n c_{1i}(1+q_{1i}^2)x_i\\
     \sum\limits_{i=s+1}^n c_{2i}(1+q_{1i}^2)x_i \\
     \vdots \\
     \sum\limits_{i=s+1}^n c_{si}(1+q_{1i}^2)x_i \\
     \end{array}
     \right).\\
\end{align*}
Let $(a_{ij})_{s\times s}= [I_{s\times s}-C_s]^{-1}$ and 
$(b_{ij})_{s\times s}= [I_{s\times s}-C_s]^{-1}[I_{s\times s}+C_s]$. 
Then for any $j\le s$, we have
$$\tau_{\phi}(x_j)= \sum\limits_{l=1}^sb_{jl}x_l + 
\sum\limits_{l=1}^s a_{jl}(\sum\limits_{i=s+1}^n c_{li}(1+q_{1i}^2)x_i) .$$
Hence the Nakayama automorphism of
$E:=A^!(Q,C)$ is given by
$$\mu_E: x_j\mapsto \begin{cases}
(-1)^n(\prod\limits_{w=1}^nq_{wj})x_j, \,\, \quad 
\qquad\qquad\qquad\qquad \qquad\qquad\qquad\qquad \;
\text{if}\,\, j> s\\
(-1)^n\sum\limits_{l=1}^s [b_{jl}q_{l1}^2(\prod\limits_{w=1}^nq_{wl})x_l 
+ a_{jl}(\sum\limits_{i=s+1}^nc_{li}(1+q_{i1}^2)(\prod\limits_{w=1}^nq_{wi})x_i)], 
\quad \text{if}\,  j\le s.
\end{cases}$$
By \cite[Theorem 4.3(2)]{RRZ2}, we have $\mu_E\mid_{E_1}=(-1)^n (\mu_A \mid_{A_1})^*.$ 
Thus the Nakayama automorphism of $A(Q,C)$ is given by
$$\mu_A: t_j\mapsto \begin{cases}
\sum\limits_{i=1}^s\sum\limits_{l=1}^sa_{il}c_{lj}(1+q_{j1}^2)
(\prod\limits_{w=1}^nq_{wj})t_i +(\prod\limits_{w=1}^nq_{wj})t_j, \,\, \quad \text{if}\,\, j> s\\
\sum\limits_{i=1}^sb_{ij}q_{j1}^2(\prod\limits_{w=1}^nq_{wj})t_i,\,\, \quad
\qquad\qquad\qquad\qquad\quad \;\;\text{if}\,\,  j\le s.
\end{cases} $$
\end{proof}

Graded skew Clifford algebras were introduced and studied by 
Cassidy-Vancliff in \cite{CV1, CV2}. Let $\mu:=(\mu_{ij}) 
\in {M}_n(\Bbbk)$ such that
$\mu_{ii}=1$ and $\mu_{ij}\mu_{ji}=1$ for all $1\leq i,j\leq n$.
A matrix $M:=(M_{ij})\in {M}_n(\Bbbk)$ is called {\it $\mu$-symmetric} if 
$M_{ij}=\mu_{ij}M_{ji}$ for all $i,j$, see \cite[Definition 1.2]{CV1}. 
Suppose that $M_1,\cdots, M_n$ are all $\mu$-symmetric 
$n\times n$-matrices. Following \cite[Definition 1.12]{CV1}, a 
{\it graded skew Clifford algebra} 
associated to $\mu$ and $M_1,\cdots, M_n$ is a graded
$\Bbbk$-algebra $A$ on degree-one generators $x_1,x_2,\cdots, x_n$
and on degree-two generators $y_1,y_2,\cdots, y_n$ with defining 
relations given by the following:
\begin{itemize}
\item
$x_ix_j+\mu_{ij}x_jx_i=\sum_{l=1}^n(M_l)_{ij}y_l$ for all
$i,j=1,2,\cdots, n$; and
\item 
the existence of a normalizing sequence
$\{r_1,\cdots,r_n\}$ of $A$ that spans $\Bbbk y_1+\cdots +\Bbbk y_n$.
\end{itemize}

Cassidy-Vancliff \cite{CV1, CV2} associated the geometric data to a
graded skew Clifford algebra by using noncommutative
algebraic geometry developed in \cite{ATV1,ATV2}. They generalized
the notion of a quadric and a quadric system in commutative algebra
to the noncommutative case, and proved that a graded skew Clifford
algebra is Artin-Schelter regular if and only if its associated 
quadratic system is normalizing and base point free.

\begin{proposition}
\label{yypro1.5}
Let $\Omega:=\sum_{i=1}^s t_i^2$. Then $\Omega$ is a regular
normal element of $A:=A(Q,C)$ and $A$ is a graded skew 
Clifford algebra.
\end{proposition}

\begin{proof} Let $B=\Bbbk_{q_{ij}}[t_1,\cdots, t_n]$ and also consider
$\Omega$ as an element in $B$. By \eqref{E0.0.3},
$t_j \Omega=q_{1j}^2 \Omega t_j$ for all $j$. Hence, $\Omega$ is a 
regular normal element in $B$, and   
the Hilbert series of $B/(\Omega)$ is $\frac{(1-t^2)}{(1-t)^{n}}$.
By comparing the relations, one sees that
$$A/(\Omega)\cong B/(\Omega).$$
Hence we have a surjective map
$$A/\Omega A\xrightarrow{ f} A/(\Omega)\cong B/(\Omega).$$
By Theorem \ref{yythm0.2}(2), $A$ is a domain. Then
$\Omega\in A$ is regular and 
the Hilbert series of $A/\Omega A$ is $\frac{(1-t^2)}{(1-t)^{n}}$,
which equals the Hilbert series of $B/(\Omega)$. 
This implies that $\Omega A=(\Omega)\subseteq A$. By symmetry,
$A\Omega=(\Omega)$. Therefore $\Omega$ is a normal element 
in $A$. 

Now let $\Omega_1=\Omega$ and $\Omega_i=t_i^2$ for $i=2,\cdots,n$.
After we showed that $\Omega_1$
is a normal element in $A$, one sees that 
$\{\Omega_1,\Omega_2,\cdots,\Omega_n\}$ is a normalizing 
sequence of $A$, which plays the role of $\{r_1,\cdots,r_n\}$
in \cite[Definition 1.12]{CV1}. 
Then $t_i$s play the role of $x_i$s and $t_i^2$s play the role
of $y_i$ in \cite[Definition 1.12]{CV1}.
Let $\mu$ be $Q$. Then one can easily recover 
the set of matrices $M_1,\cdots, M_n$ by the relations of $A$.
Therefore $A$ is a graded skew Clifford algebra by 
\cite[Definition 1.12]{CV1}.
\end{proof}

\begin{proposition}
\label{yypro1.6}
Let $s<n$.
Let $Q_s:=(q_{ij})_{s\times s}$ and $C_s:=(c_{ij})_{s\times s}$
and let $A(Q_s, C_s)$ be the subalgebra of $A(Q,C)$
generated by $\{t_1,\cdots,t_s\}$. Then $A(Q_s,C_s)$ 
is Artin-Schelter regular and $A(Q,C)$ is an iterated 
Ore extension of $A(Q_s,C_s)$.
\end{proposition}

\begin{proof} Note that \eqref{E0.0.1}-\eqref{E0.0.8}
hold for the submatrices $(Q_s, C_s)$. The first assertion 
follows from Theorem \ref{yythm0.2}. For each $j>s$, using 
the relations
$$t_j t_i= q_{ij} t_i t_j +c_{ij} (\sum_{i=1}^s t_s^2), \quad
\forall \; i<j,$$
one can easily check that $A(Q,C)$ is an iterated 
Ore extension of $A(Q_s,C_s)$ by adding $t_{s+1}, t_{s+2},
\cdots, t_n$ consecutively.
\end{proof}

It is possible that $A(Q,C)$ in Theorem \ref{yythm0.2} is 
always an iterated Ore extension from $\Bbbk$, which can be
verified for all examples in the next section.

\section{Examples}
\label{yysec2}
In this section we use the main result to construct some Artin-Schelter
regular algebras of low global dimension and to calculate their Nakayama 
automorphisms. 

\begin{example}\label{yyex2.1}
Let $n=3$, $s=1$ and 
$Q=\left(
    \begin{array}{ccc}
    1 & 1 & 1\\
    1 & 1 & q\\
    1 & q^{-1} & 1
    \end{array}
    \right) \quad {\rm{and}}\quad  
C=\left(
    \begin{array}{ccc}
    0 & 0 &  0\\
    0 & 0 &  1\\
    0 & -q^{-1}& 0 
    \end{array}
    \right).$
All conditions \eqref{E0.0.1}-\eqref{E0.0.8} are 
satisfied. The algebra is isomorphic to 
$$\Bbbk\langle t_1,t_2,t_3\rangle
( t_3t_2-qt_2t_3-t_1^2,\;\; t_1\; {\text{central}}).$$ 
This is an Artin-Schelter regular of dimension 3, which is 
the algebra $A(2)$ in \cite{LMZ}. The Nakayama automorphism 
of $A(Q,C)$ is given by
$$\mu_A: t_1\mapsto t_1,\quad t_2 \mapsto q^{-1} t_2,
\quad t_3\mapsto 
q t_3.$$
When $q$ is not a root of unity, the
group actions, Hopf algebra actions and cancellation property of this 
algebra was studied in \cite{LMZ}. 
\end{example}

When $n=4$, there are many solutions to system
\eqref{E0.0.1}-\eqref{E0.0.8} (for example, solutions can be listed by Maple). 
We pick two of them  and use Theorem \ref{yythm0.2} to calculate their Nakayama
automorphism. It is possible that these algebras are isomorphic,
in a non-obvious way, to some algebras constructed and studied in 
\cite{CV1, CV2, NV, VV}.

\begin{example}\label{yyex2.2}
Let $n=4$, $s=2$,
$Q=\left(
        \begin{array}{cccc}
        1 & -1 & \frac{1}{b}& -1\\
       -1 & 1 & \frac{1}{b} &1\\
        b & b& 1& b \\
        -1& 1& \frac{1}{b}&1
        \end{array}
        \right)$ 
and $C=\left(
        \begin{array}{cccc}
        0 & 0 & 0& 0\\
        0 & 0 & 0 &a\\
        0 & 0 & 0& 0 \\
        0 & -a& 0&0
        \end{array}
        \right).$
Then  $$A(Q,C)=  \frac{\Bbbk\langle t_1,t_2,t_3,t_4 \rangle}
{\left(
        \begin{array}{c}
        t_2t_1+t_1t_2\\
        bt_3t_1-t_1t_3\\
        t_4t_1+t_1t_4\\
        bt_3t_2-t_2t_3 \\
        t_4t_2-t_2t_4-a(t_1^2+t_2^2) \\
        t_4t_3-bt_3t_4
        \end{array}
        \right)}. $$
We have $$(I_{2\times 2}-C_2)= 
\left(
        \begin{array}{cc}
        1 & 0  \\
        0 & 1  \\
        \end{array}
        \right)=I_{2\times 2},
\,\,\,\,(I_{2\times 2}-C_2)^{-1}= I_{2\times 2} $$
and 
$$ (I_{2\times 2}-C_2)^{-1}\cdot (I_{2\times 2}+C_2)= I_{2\times 2}.$$
By Theorem \ref{yythm0.2}(3), the Nakayama automorphism is given by
$$\mu_A: \quad
t_1\mapsto bt_1, \quad t_2\mapsto -b t_2, \quad
t_3\mapsto \frac{1}{b^3} t_3, \quad
t_4\mapsto -2abt_2-bt_4.$$
\end{example}

\begin{example}\label{yyex2.3}
Let $n=4$, $s=3$ and 
$Q=\left(
   \begin{array}{cccc}
   1 & -1 & 1& \frac{1}{b}\\
  -1 & 1 & -1 &\frac{1}{b}\\
   1 & -1& 1& \frac{1}{b} \\
   b & b& b&1
   \end{array}
  \right)$ and $C=\left(
  \begin{array}{cccc}
   0 & 0 & a & 0\\
   0 & 0 & 0 & 0\\
  -a & 0 & 0 & 0 \\
   0 & 0 & 0 & 0
  \end{array}
  \right)$ with $1+a^2\neq 0, a\in \Bbbk$ and $b\in \Bbbk^{\times}$.
Then $$A(Q,C) =  
\frac{\Bbbk\langle t_1,t_2,t_3,t_4 \rangle}
{\left(
    \begin{array}{c}
    t_2t_1+t_1t_2\\
    t_3t_1-t_1t_3-a(t_1^2+t^2_2+t_3^2)\\
    bt_4t_1-t_1t_4\\
    t_2t_3+t_3t_2 \\
    bt_4t_2-t_2t_4 \\
    bt_4t_3-t_3t_4
    \end{array}
    \right)}. $$ 
We have  
$$(I_{3\times 3}-C_3)= 
    \left(
    \begin{array}{ccc}
    1 & 0 & -a \\
    0 & 1 & 0 \\
    a & 0 & 1 \\
    \end{array}
    \right), \,\,\,\,
(I_{3\times 3}-C_3)^{-1}= 
    \left(
    \begin{array}{ccc}
    \frac{1}{1+a^2} & 0 & \frac{a}{1+a^2} \\
    0 & 1 & 0 \\
    \frac{-a}{1+a^2} & 0 & \frac{1}{1+a^2} \\
    \end{array}
    \right) $$
and $$ (I_{3\times 3}-C_3)^{-1}\cdot (I_{3\times 3}+C_3)
= \left(
    \begin{array}{ccc}
    \frac{1-a^2}{1+a^2} & 0 & \frac{2a}{1+a^2} \\
    0 & 1 & 0 \\
    \frac{-2a}{1+a^2} & 0 & \frac{1-a^2}{1+a^2} \\
    \end{array}
    \right).$$ 
		By Theorem \ref{yythm0.2}, its Nakayama automorphism is given by
$$\mu_A: t_j\mapsto \begin{cases}
\frac{a^2b-b}{1+a^2}t_1+\frac{2ab}{1+a^2}t_3,\,\, \,\, \qquad \qquad \text{if}\,\, j=1\\
bt_2,  \qquad\qquad\qquad\qquad \quad \;\; \text{if}\,\, j=2\\
\frac{-2ab}{1+a^2}t_1 + \frac{a^2b-b}{1+a^2}t_3 ,  \,\,\,\, \qquad\qquad \text{if}\,\, j=3\\
\frac{1}{b^3}t_4, \qquad\qquad\qquad\qquad\quad \;\, \text{if}\,\, j=4.
\end{cases} $$
\end{example}

\begin{example}
\label{yyex2.4}
Let $s=n$, $q_{ij}=1$ for all $i,j$, and $C$ be a skew 
symmetric matrix.  Then \eqref{E0.0.1}-\eqref{E0.0.5} are
trivially satisfied. Equation \eqref{E0.0.8} is saying that
\begin{equation}
\label{E2.4.1}\tag{E2.4.1}
\det(I_{n\times n}-C)\neq 0.
\end{equation}
It is not hard to check that \eqref{E0.0.6} is equivalent to
\eqref{E0.0.7}, and equivalent to 
\begin{equation}
\label{E2.4.2}\tag{E2.4.2}
c_{ij} c_{kl}+c_{jk}c_{il}-c_{ik}c_{jl}=0,
\quad \forall\;\; 
1\leq i<j<k<l\leq n.
\end{equation}
Let $A(Q,C)$ be the algebra defined by 
\begin{equation}
\label{E2.4.3}\tag{E2.4.3}
A(Q,C)=\Bbbk\langle t_1,\cdots,t_n\rangle/(t_j t_i -t_it_j-
c_{ij} (\sum_{l=1}^n t_l^2), \forall i < j).
\end{equation}
When $C$ satisfies \eqref{E2.4.1} and \eqref{E2.4.2},
by Theorem \ref{yythm0.2}, $A(Q,C)$ is a Koszul, strongly 
noetherian, Auslander regular algebra of global dimension 
$n$ and its Nakayama automorphism is determined by
$$\mu_A:t_j\mapsto \sum\limits_{i=1}^nb_{ij}t_i, j=1,2,\cdots, n,$$ 
where  $(b_{ij})_{n\times n}=(I_{n\times n}-C)^{-1}(I_{n\times n}+C)$.

It is clear that $(b_{ij})_{n\times n}$ is the identity
matrix if and only if $C$ is zero. Therefore $A(Q,C)$ is 
Calabi-Yau if and only if $C=0$ if and only if 
$A(Q,C)=\Bbbk[t_1,\cdots,t_n]$.

Note that the algebra $A(Q,C)$ is a deformation quantization of the 
Poisson algebra given in \cite[Section 3.2]{LWZ}.
By \cite[Lemma 3.2]{LWZ}, \eqref{E2.4.2} is equivalent to 
${\text{rank}} \; C\leq 2$. When $\Bbbk$ is algebraically closed,
it follows from \cite[Theorems 2.1 and 2.5]{DRZ} that
every skew symmetric matrix $C$ with rank at most 2 (when $n=4$) 
is orthogonally similar to either
$$\left(
   \begin{array}{cccc}
   0 & a & 0 & 0\\
  -a & 0 & 0 & 0\\
   0 & 0 & 0 & 0\\
   0 & 0 & 0 & 0
   \end{array}
  \right), 
	\quad {\text{or}} \quad
	\left(
  \begin{array}{cccc}
   0 & i& 0 & 0\\
   -i & 0 & 1 & 0\\
  0 & -1 & 0 & 0 \\
   0 & 0 & 0 & 0
  \end{array}
  \right),
	\quad {\text{or}} \quad
	\left(
  \begin{array}{cccc}
   0 & -1 & i & 0\\
   1 & 0 & 0 & -i\\
  -i & 0 & 0 & -1 \\
   0 & i & 1 & 0
  \end{array}
  \right)$$
where $a\in \Bbbk$ and $i^2=-1$. This implies that 
there are only three isomorphism classes of $A(Q,C)$ corresponding to 
the above three matrices respectively (the assertion is also valid 
when $n>4$).
In all cases, one can show that $A(Q,C)$ is an iterated Ore extension 
from $\Bbbk$ (details are omitted). 
\end{example}

\providecommand{\bysame}{\leavevmode\hbox to3em{\hrulefill}\thinspace}
\providecommand{\MR}{\relax\ifhmode\unskip\space\fi MR }
\providecommand{\MRhref}[2]{%
\href{http://www.ams.org/mathscinet-getitem?mr=#1}{#2} }
\providecommand{\href}[2]{#2}

\end{document}